\documentclass[12pt]{amsart}
\usepackage{amsmath}
\usepackage{amsfonts}
\usepackage{mathrsfs}
\usepackage{amssymb}
\usepackage{amsthm}
\usepackage{tikz}
\usepackage{pgf}

\newtheorem{theorem}{Theorem}[section]
\newtheorem{proposition}[theorem]{Proposition}
\newtheorem{lemma}[theorem]{Lemma}

\newtheorem{corollary}[theorem]{Corollary}

\theoremstyle{definition}
\newtheorem{definition}[theorem]{Definition}

\def\C{\mathbb C}
\def\Z{{\mathbb Z}}
\def\R{{\mathbb R}}
\def\P{{\bf P}}
\def\D{{\mathcal D }}
\def\Q{{\bf Q}}
\def\T{{\bf T}}
\def\I{{\mathcal I}}

\def\size{{\rm size}}
\def\diam{{\rm diam}}

\def\K{J}  % dyadic scale separation

\def\pv{{\vec P}}

\def\Pv{{\vec{\bf P}}}
\def\dist{{\rm dist}}
\def\S{{S}}

\def\111{\gamma}

\def\be#1{\begin{equation}\label{#1}}
\def\bas{\begin{align*}}
\def\eas{\end{align*}}
\def\bi{\begin{itemize}}
\def\ei{\end{itemize}}

\def\emph#1{{\it #1}}

\begin{document}
\title{Bilinear Fourier restriction theorems}

\author{Ciprian Demeter}
\address{Department of Mathematics, Indiana Unversity, Bloomington, IN 47405}
\email{demeterc@indiana.edu}
\thanks{}

\author{S.\ Zubin Gautam}
\address{Department of Mathematics, Indiana Unversity, Bloomington, IN 47405}
\email{sgautam@indiana.edu}
\thanks{}

%\subjclass[2000]{Primary }
%    For articles to be published after 1 January 2010, you may use
%    the following version:


\date{}

\dedicatory{}

\begin{abstract}
We provide a general scheme for proving $L^p$ estimates for certain bilinear Fourier restrictions outside the locally $L^2$ setting. As an application, we show how such estimates follow for the lacunary polygon. In contrast with prior approaches, our argument avoids any use of the Rubio de Francia Littlewood--Paley inequality.
\end{abstract}

\maketitle

\section{Introduction}\label{introduction}
Given a domain $D$ in $\R \times \R$, one may consider the associated ``bilinear Fourier restriction'' operator $T_D$, defined \textit{a priori} on pairs of test functions $(f,g) \in \mathcal S(\R) \times \mathcal S(\R)$ and given by \[T_D\big(f,g\big)(x) = \int_{\R^2} \widehat f(\xi) \, \widehat g(\eta) \, \chi_D (\xi, \eta) \, e^{2\pi i x\,(\xi + \eta)} \, \mathrm d\xi \, \mathrm d\eta.\]  That is, $T_D$ is the bilinear Fourier multiplier operator with symbol $\chi_D$, the characteristic function of $D \subseteq \R \times \R$.  The case of $D = \R \times \R$ simply yields the pointwise product operator $(f,g) \mapsto fg$, and the well-known bilinear Hilbert transforms $H_s$, given by \[H_{s}(f,g)(x)= \int f(x+st) \, g(x-t)\,\frac{{\mathrm d}t}{t},\] are essentially bilinear Fourier restrictions to half-planes in $\R \times \R$, modulo linear combinations with the pointwise product operator.  The main topic of this paper concerns the boundedness properties of such bilinear F
 ourier restriction operators from $L^{p_1} (\R) \times L^{p_2}(\R)$ to $L^{p_3} (\R)$, for exponent triples $(p_1 \, , \, p_2 \, , \, p_3)$ satisfying the H\"older homogeneity condition $\frac{1}{p_1} + \frac{1}{p_2} + \frac{1}{p_3} = 1$.

The case of $D = \mathbb D$ the unit disc in $\R \times \R$ was studied in \cite{GL} by Grafakos and Li, who proved the boundedness of $T_{\mathbb D}$ in the ``locally $L^2$'' range of exponent triples satisfying $2 \leq p_1 \, , \, p_2 \, , \, p_3 \leq \infty$.  Such a nontrivial boundedness result of course contrasts markedly with the scenario for \emph{linear} Fourier restriction operators to domains in $\R^2$; indeed, by a celebrated result of C.~Fefferman, the characteristic function of the disc is well known to yield a bounded Fourier multiplier operator on $L^p(\R^2)$ only in the trivial case $p=2$.  In light of this contrast, at first glance one might speculate that in the bilinear setting the locally $L^2$ range of boundedness obtained in \cite{GL} may play a similar role to that of $L^2$ in the linear setting; indeed, to date no boundedness result for $T_{\mathbb D}$ has been obtained outside the locally $L^2$ range.

However, upon further examination such an analogy seems unlikely to obtain.  To wit, boundedness estimates for the bilinear Hilbert transforms $H_s$ \emph{uniform} in the parameter $s$ have been established outside the locally $L^2$ range (cf.~\cite{GL2}, \cite{Li2}).  By a classical argument invoking invariance of bilinear multiplier norms under dilations and translations of the symbol (cf.~\cite{GL}), boundedness of $T_{\mathbb D}$ immediately yields boundedness estimates for the bilinear Hilbert transforms $H_s$ uniformly in the parameter $s$.  Conversely, the core argument for the boundedness of $T_{\mathbb D}$ in \cite{GL} consists mainly of decomposing the disc multiplier in a suitable manner to facilitate the application of precisely the techniques used to prove uniform bounds for the $H_s$ in \cite{Thieleuniform}, \cite{GL2}, and \cite{Li2}.  There is as yet no argument that can deduce boundedness of $T_{\mathbb D}$ in a given exponent range by using uniform boundedne
 ss of the $H_s$ as a ``black box'' result; nonetheless, it seems likely that the exponent range for boundedness of the bilinear disc multiplier should be identical to that for uniform boundedness of the bilinear Hilbert transforms, and the obstacles to extending results for the disc beyond the locally $L^2$ range seem to be more technical than fundamental in nature.

The purpose of the current paper is to advocate for this point of view by illustrating in a simplified setting how one can circumvent one of the main such technical obstacles.  Namely, the key feature of the Grafakos--Li argument in \cite{GL} that limits its scope to the locally $L^2$ setting is the use of Rubio de Francia's Littlewood--Paley inequality to treat certain ``error terms'' arising from the decomposition of the disc multiplier therein; this method should be viewed as an exploitation of $L^2$ orthogonality.  In the current paper we follow the approach developed by Muscalu, Tao, and Thiele in \cite{MTT} for establishing uniform estimates on the $H_s$ (as opposed to the approach in \cite{Thieleuniform}, \cite{GL2}, and \cite{Li2}), and we exploit orthogonality in a manner that is not restricted to the locally $L^2$ setting.

Theorem \ref{main3} below establishes boundedness of the operator $T_D$ for a certain range of exponents outside the locally $L^2$ setting, with $D$ the ``lacunary polygon''; the lacunary polygon multiplier captures the most germane features of the disc multiplier in relation to uniform estimates for the bilinear Hilbert transforms.  In order to minimize technicalities and clearly illustrate the principles at work, we have not treated the actual disc multiplier $T_{\mathbb D}$ explicitly in this paper; however, with some additional technical effort but essentially the same ideas, it is likely that this approach could yield boundedness for $T_{\mathbb D}$ outside the locally $L^2$ range.  Theorem \ref{main2} illustrates the underlying mechanism behind this boundedness result, namely that an arbitrary sum of bilinear Hilbert transforms localized to have disjoint ``frequency supports'' can be bounded.  Both of these results will follow after we have established a certain ``model
  sum'' estimate, Theorem \ref{main1} below.  We obtain this latter estimate by following the Muscalu--Tao--Thiele approach of \cite{MTT}, with some minor modifications that allow their estimates to be extended to a limited range outside the locally $L^2$ setting.  This approach lends itself to our method of utilizing the orthogonality afforded by our disjoint-frequency-support conditions; in essence, the orthogonality allows us to aggregate the ``Bessel-type'' estimates associated to a family of bilinear Hilbert transforms into a single ``global'' Bessel-type estimate for their sum.  Our results are all stated for exponent ranges outside the locally $L^2$ setting; however, we note that our methods apply equally well in that setting.

As a final remark, we note that Oberlin and Thiele (\cite{OberlinThiele}) have obtained uniform estimates for the \emph{quartile operators}, which are Walsh model analogs of the bilinear Hilbert transforms, in the full range of expected exponents.  Via the ``global Bessel'' approach of this paper, their approach to uniform estimates can be used to prove bounds for a suitable Walsh analog of the disc and lacunary polygon multipliers in the full Banach range of exponents; one would expect that an adaptation of their methods to the Fourier setting (that is, the setting of the genuine bilinear Hilbert transforms) could similarly be used to obtain boundedness of the lacunary polygon and disc multipliers in the full expected range.

\section{The main theorems}
\label{uniform}

Let us first describe the boundedness result for the bilinear ``lacunary polygon'' multiplier, which should be viewed as a simplified model for the bilinear disc multiplier.  Let $P_{\operatorname{lac}}$ be the lacunary polygon inscribed in the unit disk; its vertices are the points $\big(\cos(\pi2^{-\mu}),\sin(\pi 2^{-\mu})\big)$, $\mu\ge 1$, and their symmetric images in the remaining three quadrants (see Figure \ref{lacpolyfigure}). A salient feature of this region is that the slopes of the edges of the polygon are bounded away from zero.

\begin{figure}
\begin{tikzpicture}[scale=.75]
\draw [<->] (-5,0) -- (5,0);
\draw [<->] (0,-5) -- (0,5);
\draw [help lines] (0,0) circle (4);
\draw [thick] (90:4) -- (45:4) -- (22.5:4) -- (11.25:4)--(5.625:4)--(2.8125:4)--(0:4);
\draw [help lines] (0,0) -- (45:4);
\draw [help lines] (0,0) -- (22.5:4);
\draw [help lines] (0,0) -- (11.25:4);
\draw [dotted, thick] (8:3.5) -- (3:3.5);
\draw [thick, cm={-1,0,0,1,(0,0)}](90:4) -- (45:4) -- (22.5:4) -- (11.25:4)--(5.625:4)--(2.8125:4)--(0:4);
\draw [thick, cm={1,0,0,-1,(0,0)}](90:4) -- (45:4) -- (22.5:4) -- (11.25:4)--(5.625:4)--(2.8125:4)--(0:4);
\draw [thick, cm={-1,0,0,-1,(0,0)}] (90:4) -- (45:4) -- (22.5:4) -- (11.25:4)--(5.625:4)--(2.8125:4)--(0:4);
\fill (45:4) circle(1.5pt);
\draw (45:4) node [font=\scriptsize, right] {$v_\mu$};
\fill (22.5:4) circle(1.5pt);
\draw (22.5:4) node [font=\scriptsize, right] {$v_{\mu+1}$};
\fill (11.25:4) circle(1.5pt);
\draw [help lines] (45:2) arc (45:22.5:2);
\draw (32.5:3) node [font=\tiny] {$\pi \cdot 2^{-(\mu+1)}$};
\end{tikzpicture}
\caption{The lacunary polygon $P_{\operatorname{lac}}$.} \label{lacpolyfigure}
\end{figure}
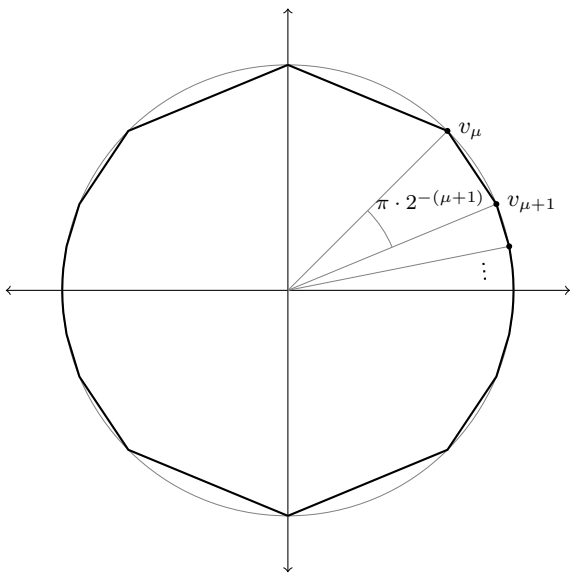

\begin{theorem}
\label{main3}
The bilinear operator $H_{\operatorname{lac}}$ given by
$$H_{\operatorname{lac}}\big(f,g\big)(x):=\int_{\R^2} \widehat{f}(\xi)\, \widehat{g}(\eta) \; \chi_{P_{\operatorname{lac}}} (\xi,\eta) \; e^{2\pi ix(\xi+\eta)} \, {\mathrm d}\xi \, {\mathrm d}\eta$$
maps $L^{p_1}\times L^{p_2}\to L^{p_3'}$
whenever $1<p_1<2<p_2,p_3<\infty$ and $$\frac1{p_1}+\frac1{p_2}+\frac1{p_3}=1.$$
\end{theorem}

As alluded to in the introduction above, our second result will illustrate the actual principle at work behind Theorem \ref{main3}.  Let us begin by constructing a family of bilinear Hilbert transforms that are ``frequency localized with disjoint frequency supports.''  Consider the family of the bilinear Hilbert transforms defined as before by
$$H_{s}(f,g)(x)= \int f(x+st) \, g(x-t)\,\frac{{\mathrm d}t}{t}.$$
It has been proven in \cite{GL2} and \cite{Li2}  that $H_{s}$ maps $L^{p_1}\times L^{p_2}$ to $L^{p_3'}$ uniformly in the parameter $s$, if $\frac1{p_3}+\frac1{p_3'}=1$ and $\left(\frac{1}{p_1},\frac{1}{p_2},\frac{1}{p_3}\right)$ is in the convex hull of the triangles $c$, $a_2$, and $a_3$ in the ``type diagram'' of Figure \ref{typediagram}.

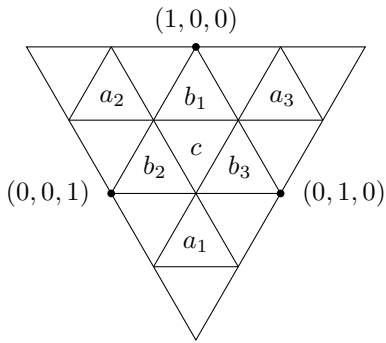
\begin{figure}
\begin{tikzpicture}[scale=1.95]
\draw (-1.1547,1) -- (1.1547,1) -- (0,-1) -- cycle;
\fill (-.5774,0) circle(.75pt);
\fill (0,1) circle(.75pt);
\fill (.5774,0) circle(.75pt);
\draw (-.5774,0) -- (0,1) -- (.5774,0) -- cycle;
\draw (-.8660,.5) -- (.8660,.5);
\draw (-.5774,1) -- (.2887,-.5);
\draw (.5774,1) -- (-.2887,-.5);
\draw (-.2887,-.5) -- (.2887,-.5);
\draw (-.5774,1) -- (-.8660,.5);
\draw (.5774,1) -- (.8660,.5);
\draw (0,.3) node [font=\footnotesize] {$c$};
\draw (0,.65) node [font=\footnotesize] {$b_1$};
\draw (-.28,.18) node [font=\footnotesize] {$b_2$};
\draw (.30,.18) node [font=\footnotesize] {$b_3$};
\draw (0,-.35) node [font=\footnotesize] {$a_1$};
\draw (-.57,.65) node [font=\footnotesize] {$a_2$};
\draw (.59,.65) node [font=\footnotesize] {$a_3$};
\draw (0,1.18) node [font=\footnotesize] {$(1,0,0)$};
\draw (-.65,0) node [font=\footnotesize, left] {$(0,0,1)$};
\draw (.65,0) node [font=\footnotesize, right] {$(0,1,0)$};
\end{tikzpicture}
\caption{``Type diagram'' of reciprocal exponent triples $\left(\frac{1}{p_1} \, , \, \frac{1}{p_2} \, , \, \frac{1}{p_3}\right)$.} \label{typediagram}
\end{figure}

Let $(I_\mu^i)_{\mu\in\Z}$, $i=1,2,3$, be three families of intervals in $\R$ such that the $I^{i}_\mu$ are pairwise disjoint for each $i$, and such that for each $\mu$ we have $I^3_\mu=-I_\mu^1-I_{\mu}^2$. Assume in addition that $s_\mu=\frac{|I_\mu^2|}{|I_\mu^1|}\ge 1$; thus in particular $|I_\mu^1|$ is the shortest among the three intervals $I_\mu^i$.\footnote{It will be clear from the proof that the condition $s_\mu\geq 1$ can be replaced with $s_\mu\gtrsim 1$.}

Let $\varphi_{\mu}^i$ be a fixed function adapted to and supported on $I_{\mu}^i$.
For each suitable function $f^i$ on $\R$ we define the smooth Fourier restriction $f^i_\mu$ of $f^i$ to $I^i_\mu$ via $$\widehat{f^i_\mu}=\widehat{f^i}\varphi_\mu^i.$$  Define the operator $M_a$ of modulation by $a$ via $$\widehat{M_a(f)}(\xi)=\widehat{f}(\xi+a).$$

For an arbitrary line $l_\mu$ in the plane with slope $s_\mu$, define the line $l_\mu^3$  in $\R^3$ by
$$l_\mu^3 = \{(\xi,\eta,\theta):(\xi,\eta)\in l_\mu,\text{ and }\theta=-\xi-\eta\}.$$ Given $(\xi_\mu,\eta_\mu,\theta_\mu)\in l_\mu^3$, define a trilinear form $\Lambda_\mu$ via
\[\Lambda_\mu(f^1,f^2,f^3)=\int H_{s_\mu}\big(M_{\xi_\mu}f^1_\mu,M_{\eta_\mu}f^2_\mu\big)(x) \; M_{\theta_\mu}{f}^3_\mu(x)\, {\mathrm d}x,\] and note that in fact the definition does not depend on the particular choice of  $(\xi_\mu,\eta_\mu,\theta_\mu)$.

\begin{theorem}
\label{main2}
The  tri-sublinear form
$$(f^1,f^2,f^3)\mapsto \sum_\mu |\Lambda_\mu(f^1,f^2,f^3)|$$
is bounded on  $L^{p_1}(\R) \times L^{p_2}(\R) \times L^{p_3}(\R)$
whenever $1<p_1<2<p_2,p_3<\infty$ and $$\frac1{p_1}+\frac1{p_2}+\frac1{p_3}=1.$$
\end{theorem}

Thus, our result establishes boundedness in the triangle $b_1$ in the type diagram of Figure \ref{typediagram}.  In the locally $L^2$ region $c$ of the type diagram in Figure \ref{typediagram}, \textit{viz.}~for $2<p_1,p_2,p_3<\infty$, the boundedness of $\sum_\mu |\Lambda_\mu(f^1,f^2,f^3)|$ follows immediately from the uniform estimates for $H_{s}$ from \cite{Thieleuniform}, combined with  Rubio de Francia's Littlewood--Paley inequality; see for example Lemma 1 of \cite{GL}. The use of the Rubio de Francia inequality is restricted to this locally $L^2$ range, and our argument shows how to avoid it when dealing with larger ranges of exponents.\footnote{Again, while our results are stated only for $1<p_1<2<p_2,p_3<\infty$, the methods below apply equally well in the locally $L^2$ setting.} Our approach exploits the orthogonality of the functions $f^i_\mu$ in a slightly more subtle way, by turning ``local'' Bessel-type estimates into a ``global'' Bessel inequality. We rely heav
 ily on the tools developed in \cite{MTT}, and in doing so we attempt to keep technicalities to a minimum. It is plausible that the same idea of obtaining a global Bessel inequality could further extend the range of exponents in Theorems \ref{main3} and \ref{main2}, but we will not pursue that here.

We now lay the groundwork for our third result.  First, we recall some notation introduced in \cite{MTT}, which we shall adapt to the current context.  The purpose of this discussion is to develop certain ``model sum'' forms for the forms and operators appearing in Theorems \ref{main2} and \ref{main3}; by a standard time-frequency discretization procedure, we will eventually reduce the boundedness results of those theorems to the boundedness of such model sum forms, which is established in Theorem \ref{main1} below.  Cf.~Sections 2 and 3 of \cite{MTT}; our notation is largely identical to that appearing therein.

Define the diagonal \[\tilde \Gamma' =\{(\xi,\xi,\xi):\xi\in \R\}.\]  For each $\mu\in\Z$, let ${\bf v}_\mu=(1,s_\mu,-1-s_\mu)$ with $s_\mu\ge 1$. Note that in \cite{MTT} it is the third component of ${\bf v}_\mu$ that is the shortest (and normalized to 1), while in our case it will always be the first component. This will account for the asymmetry in the range of exponents appearing in our results.  Let $L_\mu: \R^3 \to \R^3$ be the linear transformation given by $$L_\mu(\xi,\eta,\theta)=\big(\xi,s_\mu\eta,(-1-s_\mu)\theta\big).$$

Following \cite{MTT}, we choose a sufficiently large but unspecified $N$ depending on $p_1,p_2,p_3$ which will encode the decay of functions in physical space; we will refer to $N$ as the \emph{spatial parameter}. We will also need some $C_0\ge 2^{300}$; at various points we may need to increase the values of $N$ and $C_0$ to accommodate our argument.  As in \cite{MTT}, we define $J:=2^{C_0}$.

We denote by $\tilde \Q$ the collection of all cubes $\tilde Q$ in $\R^3$ with dyadic side-lengths $2^{j}$, centered in $2^{j-10} \Z^3$, and satisfying the Whitney conditions\footnote{We have a 10 in \eqref{whitney2first}, while in \cite{MTT} there is a 4. The argument in \cite{MTT} does not change.}
\begin{equation}\label{whitney1first}
 C_0 \, \tilde Q\cap \tilde\Gamma '=\emptyset,
\end{equation}
\begin{equation}\label{whitney2first}
 10\, C_0 \, \tilde Q\cap \tilde\Gamma '\neq \emptyset.
\end{equation}
For each $\mu \in \Z$ we  introduce the collection of boxes $$\Q_\mu:=\{L_\mu(\tilde Q):\tilde Q\in \tilde \Q\}$$
and  the collection $\Pv_1^{\mu}$ of \emph{multi-tiles}. Each multi-tile $\pv =(P_1,P_2,P_3)\in \Pv_1^{\mu}$ is identified by a spatial dyadic interval $I_{\pv}$ and a ``frequency box'' $Q_{\pv} = \omega_{P_1}\times \omega_{P_2}\times\omega_{P_3}\in \Q_\mu$, such that $|I_{\pv}||\omega_{P_1}|=1$. We refer to the components $P_i:=I_{\pv}\times \omega_{P_i}$ of $\pv$ as  \emph{$(i,\mu)$-tiles}, or, if no confusion arises, as \emph{$i$-tiles} or  simply as \emph{tiles}.  When we refer to tiles $P$ without explicit reference to multi-tiles, they will be notated as $P=I_P \times \omega_P$.

For a multi-tile $\pv$, we let $j_{\pv}$ denote the  number  such that $2^{-J \cdot j_{\pv}}=|I_{\pv}|$, which encodes the ``spatial scale'' of the multi-tile $\pv$. Similarly, for a box $Q = \omega_1 \times \omega_2 \times \omega_3 \in \Q_\mu$ we let $j_{Q}$ denote the  number  such that $2^{J \cdot j_{Q}}=|\omega_1|$; thus for a multi-tile $\pv$ as above, we have $j_{\pv} = j_{Q_{\pv}}$.

From now on, for an interval $\omega \subset \R$, $\pi_\omega$ will denote a multiplier operator
\begin{equation}
\label{stmultiplerety}
\widehat{\pi_\omega f}=m_\omega\widehat{f},
\end{equation}
where $m_\omega$ is a fixed function adapted to and supported on $\omega$.  In effect, such operators will be used to frequency-localize a triple of functions to the frequency box $Q_{\pv}$ of a multi-tile as above.

We would like to continue by localizing the triple of functions to the spatial interval $I_{\pv}$ of the multi-tile; so as not to destroy the frequency localization already established, the spatial localization will be achieved via a smoother version of the cutoff $\chi_{I_\pv}$.  Let $\Xi$ denote a fixed positive function with total $L^1$-mass $1$ and with Fourier transform supported in $[-2^{-2\K}, 2^{-2\K}]$, satisfying the pointwise estimates
\begin{equation}\label{eta-bounds}
 C^{-1} (1 + |x|)^{-N^2} \leq \Xi(x) \leq C (1 + |x|)^{-N^2}.
\end{equation}
Let $\Xi_j$ denote the $L^1$-normalized dilate $\Xi_j(x) := 2^{-\K j}\Xi(2^{-\K j}x)$.  For any subset $E$ of $\R$, define a ``smoothed-out'' characteristic function $\chi_{E,j}$ of $E$ by
$$ \chi_{E,j} := \chi_E * \Xi_{j}.$$  Note that for each $\mu$ and each spatial scale $j_0$ we have the natural partition of unity
\begin{equation}\label{spatialdiscret}
1 = \sum_{\pv \in \Pv_1^{\mu}\; : \; j_{\pv} = j_0} \chi_{I_\pv , j_\pv}.
\end{equation}

Both Theorem \ref{main2} and Theorem \ref{main3} will be consequences of the following; as mentioned above, it should be seen as a ``model sum estimate'' for those results.

\begin{theorem}
\label{main1}For each $\mu \in \Z$, let $I_\mu^i$ be as before.  For each given $(i,\mu$)-tile $P_i$, let $\pi_{\omega_{P_i}}$ be a multiplier operator as in \eqref{stmultiplerety}.  Then for any choice of $\xi_\mu\in I_{\mu}^1$, $\eta_\mu\in I_\mu^2$, $\theta_\mu\in I_\mu^3$, the tri-sublinear form
\begin{equation}
(f^1,f^2,f^3) \longmapsto \sum_\mu \Big | \sum_{\pv \in \Pv_1^{\mu}}
\int \chi_{I_\pv\, ,\, j_\pv} \; \pi_{\omega_{P_1}}(M_{\xi_\mu} f^1_\mu)\; \pi_{\omega_{P_2}}(M_{\eta_\mu} f^2_\mu) \; \pi_{\omega_{P_3}}(M_{\theta_\mu} f^3_\mu) \Big|
\end{equation}
is bounded on  $L^{p_1}(\R) \times L^{p_2}(\R) \times L^{p_3}(\R)$
whenever $1<p_1<2<p_2,p_3<\infty$ and $$\frac1{p_1}+\frac1{p_2}+\frac1{p_3}=1.$$
\end{theorem}

\section{Notation}
For each $1\le p<\infty$ we let
\[\mathcal{M}_pf=(\mathcal M_1(f^p))^{1/p},\] where
$\mathcal M_1 f$ is the standard Hardy--Littlewood maximal function of $f$.

For each interval $I\subset \R$ centered at $c(I)$ and with length $|I|$, and for each $\alpha>0$, we define $\alpha I$ to be the interval with length $\alpha|I|$ centered at $c(I)$. Also, if $\alpha\in\R\setminus \{0\}$ and $I=[a,b]$, then $\alpha\cdot I=[\alpha a,\alpha b]$.  These definitions extend naturally to Cartesian products of intervals in $\R^n$.

If $I=I_1\times\ldots \times I_m$, we will use the notation
$$\tilde{\chi}_I(x)=\left(1+\left|\left(\frac{x_1-c(I_1)}{|I_1|},\ldots,\frac{x_m-c(I_m)}{|I_m|}\right)\right|\right)^{-1}$$ for $x = (x_1, \ldots ,x_m) \in \R^m$.
\begin{definition}
A function $\phi:\R^m\to\C$ is said to be \emph{$L^p$ $C$-adapted of order $M$} to a box $I:=I_1\times\ldots\times I_m$ if
$$|\partial^{\alpha}\phi(x)|\le C\prod_{j=1}^{m}|I_{j}|^{-1/p-\alpha_{j}}\tilde{\chi}_{I}^M(x),$$
for each multi-index $\alpha$ with $0\leq  |\alpha| \leq M$.
\end{definition}
If $p$ is not specified, it will be implicitly understood that $p=\infty$. Occasionally, we will loosely call a function ``$L^p$-adapted to $I$'' if it is smooth and $L^p$ $C_M$-adapted of each order $M \geq 1$,  for some $C_M$ whose value will not be specified, but such that $C_M=O(1)$ for each $M$.

The implicit bounds hidden in the notation $a\lesssim b$ will be allowed to depend on the constants of adaptation and on fixed parameters such as $C_0$, $\alpha$, $N$, $M_1$, $M_2$ or $p_i$, but they will never depend on variable parameters such as $k_0$, $\mu$, $k$ or $l$. The notation $A\approx B$ will mean that $A\lesssim B$ and $B\lesssim A$.

\section{Tiles and trees}
In the next few sections we will focus on proving  Theorem \ref{main1}. We start by briefly recalling some constructions and results from \cite{MTT}. We refer the reader to \cite{MTT} for a more comprehensive account of the time-frequency analysis that we need.

First, we will make the collection $\tilde{\Q}$ of Whitney cubes sparser.
\begin{definition}
A subcollection $\tilde{\Q'}\subseteq \tilde{\Q}$ is \emph{sparse} if for any $\tilde Q = \tilde \omega_1 \times \tilde \omega_2 \times \tilde \omega_3$, $\tilde{Q'} = \tilde \omega_1' \times \tilde  \omega_2' \times \tilde  \omega_3' \in \tilde{\Q'}$ and any $1\le i\le 3$  we have:
\begin{equation}\label{sparse-scale}
|\tilde \omega_i|< |\tilde \omega_i'|\implies |\tilde \omega_i|\le 2^\K
|\tilde \omega_i'|\ ,
\end{equation}
\begin{equation}\label{sparse-single-scale}
|\tilde \omega_i|=|\tilde \omega_i'|,\
\tilde \omega_i\neq \tilde \omega_i'
 \implies \dist(\tilde \omega_i,\tilde \omega_i')\ge 2^\K\
|\tilde\omega_i|\ ,
\end{equation}
\begin{equation}\label{sparse-injectivity}
\tilde \omega_i=\tilde \omega_i'
\implies \tilde Q=\tilde{Q}'\ .
\end{equation}
\end{definition}

As in \cite{MTT}, we may decompose $\tilde \Q$ into a bounded number of sparse subcollections. Thus,
it suffices to prove Theorem \ref{main1} with $\tilde{\Q}$ replaced by a sparse subset of the original $\tilde{\Q}$; for convenience we will continue to call this sparse collection $\tilde \Q$.
Via a standard limiting argument, we also assume that $\tilde{\Q}$ is finite; the estimates we obtain will not depend on the particular finite choice of $\tilde{\Q}$. By a similar argument we can also assume as in \cite{MTT} that
$I_{\pv}\subset [-2^{Jk_0},2^{Jk_0}],$ for some fixed large $k_0$, for each $\pv\in \bigcup_\mu \Pv_1^\mu$. Thus each of the new collections $\Pv_1^\mu$ will consist of a finite number of multi-tiles. All our estimates will be independent of $k_0$.

Next, given $\tilde{Q} = \omega_1 \times \omega_2 \times \omega_3 \in \tilde{\Q}$, for each interval $1000 \, \omega_i$ define a slightly (by at most one percent on either side) largened interval $\overline {\omega}_i\supset 1000 \, \omega_i$ such that these largened intervals have the following property:  If $\tilde Q,\tilde Q'\in \tilde\Q$, $\diam(\tilde Q)<\diam(\tilde{Q}')$, and $10 \, \overline{\omega}_i\cap {\overline{\omega}'}_j\neq \emptyset$ for some $1\leq i,j\leq 3$, then $10 \, \overline{\omega}_{i'} \subseteq {\overline{\omega}'}_j$ for all $1\le i'\le 3$.  The construction of $\overline {\omega}_i$ is explained in Section 2 of \cite{MTT}; if $P=I_P\times \omega_P$ is an $i$-tile, then $\overline{\omega}_{P}$ is to be understood in the sense just described.

The following definition establishes an order relation on the set of tiles.

\begin{definition}\label{tile-order-def}
Let $P$, $P'$ be $(i,\mu)$-tiles for some $1 \leq i \leq 3$ and for a fixed $\mu$.
We say that $P \leq P'$ if $I_P \subseteq  I_{P'}$ and
$\overline{\omega}_{P} \supseteq \overline{\omega}_{P'}$.
If $\pv$ and $\pv'$ are two multi-tiles in some collection $\Pv_1^{\mu}$, we say that $\pv\le \pv'$
if $P_i\le P_i'$ for some $1 \leq i \leq 3$. Finally, we say that $\pv\lessdot \pv'$ if $\overline{\omega}_{P_i} \supseteq \overline{\omega}_{P'_i}$ for some $1 \leq i \leq 3$.
\end{definition}

It is easy to see that if $\pv\not=\pv'$ and $P_i\le P_i'$ for some $i$, then $P_j\le P_i'$ for each $j$. In particular, this makes $\le$ transitive on multi-tiles (as transitivity on tiles is clear).  It is important to realize  that we never need to and never will compare multi-tiles or tiles corresponding to distinct parameters $\mu$.

An important concept in \cite{MTT} is that of a {\em regular} set of multi-tiles, and we will adapt it to the current context.\footnote{The terminology of ``regularity'' does not appear in \cite{MTT}; our regular sets play the role of the \emph{convex} sets appearing therein.}
\begin{definition}
Let $\P\subset\Pv_1^\mu$ be a collection of multi-tiles corresponding to a fixed parameter $\mu$. For $\pv\in\P$ define the \emph{support} $E_{\pv,\P}$ of $\pv$ to be the set
\[E_{\pv \, , \, \P} := \bigcup_{\substack{\pv' \in \P, \\ Q_{\pv} = Q_{\pv'}}} I_{\pv'}.\]
Then $\P$ is called \emph{regular} if $E_{\pv\, ,\, \P}\subset E_{\pv'\, ,\, \P}$ whenever $\pv,\pv'\in \P$ and $\pv\lessdot\pv'$.
\end{definition}
Note that, by construction, each $\Pv_1^\mu$ is itself regular.

A dyadic interval is called \emph{$\K$-dyadic} if it has length
$2^{\K \cdot j}$ for some $j\in \Z$.
If $\xi=(\xi_1,\xi_2,\xi_3)\in L_\mu(\tilde\Gamma')$ and $I$ is a $\K$-dyadic interval,
we define
\begin{align*}
\omega_{1,\xi,I} &:=[\xi_1-\frac 12  |I|^{-1},\xi_1+\frac 12 |I|^{-1}]\, , \\
\omega_{2,\xi,I} &:=[\xi_2-\frac 12  s_\mu|I|^{-1},\xi_2+\frac 12 s_\mu|I|^{-1}]\, , \\
\omega_{3,\xi,I} &:=[\xi_3-\frac 12  (1+s_\mu)|I|^{-1},\xi_3+\frac 12 (1+s_\mu)|I|^{-1}]\, ,
\end{align*}
and
$$\overline{\omega}_{\xi,I}:=[L_\mu^{-1}(\xi)_i-500 |I|^{-1},L_\mu^{-1}(\xi)_i+500 |I|^{-1}]\, ,$$
where we observe that the right-hand side of the last display is actually independent of $i$.

\begin{definition}\label{tree-def} Let $\xi\in L_\mu(\tilde\Gamma')$,
let $I$ be a $\K$-dyadic interval, and let $T$ be a set of multi-tiles in some fixed collection $\vec{\textbf{P}}_1^\mu$.
The triple $(T,\xi,I)$ is called a \emph{tree} if $T \neq \emptyset$,
$I_\pv\subseteq I$ for all $\pv\in T$,
and for all $\pv\in T$ there exists $1\le i\le 3$ with
$\overline{\omega}_{\xi,I} \subseteq \overline{\omega}_{P_i}$.

We let $\omega_{i,T}$ and $\overline{\omega}_{T}$ denote $\omega_{i,\xi,I}$ and $\overline{\omega}_{\xi,I}$, respectively.  The pair $(\xi,I)$ is called the \emph{top data} of the tree. We will often refer to the set $T$ itself as a ``tree'' when explicit mention of the top data $(\xi_T,I_T)$ is not required.
\end{definition}

We emphasize that all multi-tiles in a given tree are drawn from a single collection $\Pv_1^\mu$.  Let us continue by recalling the definition of size from \cite{MTT}.\footnote{Our definition of $\size_i(T)$ introduced here is slightly smaller than the one in \cite{MTT}; the estimates \eqref{mpi-est} and \eqref{mpi-est-tree} above easily imply the corresponding estimates (33) and (35) from \cite{MTT}, since $(\xi_T)_i\in C\omega_{P_i}$ for some universal $C$. The reader can check that all the results in \cite{MTT}, in particular Proposition \ref{tree-est-prop} below, actually rely on our smaller version of size. Working with the size from \cite{MTT} would have created only some minor technical complications.}
\begin{definition}\label{size-def}
Let $1 \leq i \leq 3$, and let $f^i$ be an $L^2$ function.  Additionally, for a suitable function $\varphi$, we will let $\pi_\varphi$ denote a Fourier multiplier operator with symbol $\varphi$.

For $P_i$ an $i$-tile and $\xi_i \in \R$, define the semi-norm $\|f^i\|_{P_i,\xi_i}$ by
\be{fips}
\| f^i \|_{P_i, \xi_i} := \sup_{m_{P_i}} \big\| \tilde \chi_{I_\pv}^{10} \, \pi_{m_{P_i}}(f^i) \big\|_2;
\end{equation}
here the supremum ranges over all smooth functions $m_{P_i}$ adapted to and supported on $10 \, \omega_{P_i}$ which satisfy the estimates

\be{mpi2yfgrygf340904-est}
|m_{P_i}(\xi)| \leq \frac{|\xi - \xi_i|}{|\omega_{P_i}|}
\end{equation}
 and
\be{mpi-est}
\left|\frac{\mathrm d^km_{P_i}(\xi)}{\mathrm d\xi^k}\right| \leq \frac{1}{|\omega_{P_i}|^k}
\end{equation}
for all $\xi \in \R$ and $1 \leq k \leq N^2$.

For $T$ a tree, define the \emph{$i$-size} $\size_{i}(T)$ of $T$ with respect to $f^i$ by
\be{size-nonmax-def}
\size_{i}(T,f^i) := \bigg(\frac{1}{|I_T|} \sum_{\pv \in T} \| f^i\|_{P_i,(\xi_T)_i}^2\bigg)^{1/2}+
|I_T|^{-\frac 12}
\sup_{m_{i,T}}\big\|\tilde{\chi}_{I_T}^{10}\pi_{m_{i,T}}(f^i)\big\|_2,
\end{equation}
where the supremum ranges over all smooth $m_{i,T}$ adapted to and supported on
$10 \, \omega_{i,T}$ which satisfy the estimates

\be{mpidfnier674r4op[r4;-est-tree}
| m_{i,T}(\xi)| \leq
\frac{|\xi - (\xi_T)_i|}{|\omega_{i,T}|}\, ,
\end{equation}

\be{mpi-est-tree}
\left|\frac{\mathrm d^km_{i,T}(\xi)}{\mathrm d\xi^k}\right|\leq
\frac{1}{|\omega_{i,T}|^k}
\end{equation}
for all $\xi \in \R$ and $1 \leq k \leq N^2$.

Finally, for any collection $\Pv$ of multi-tiles, define the \emph{maximal size} $\size^*_{i}(\Pv)$ of $\Pv$ with respect to $f^i$ to be
\be{size-max-def}
\size^*_{i}(\Pv,f^i) := \sup_{(T,\xi,I):T \subseteq \Pv} \size_{i}(T,f^i)
\end{equation}
where the supremum ranges over all trees with $T\subseteq \Pv$.
\end{definition}

Most of the time the function $f^i$ with respect to which the size is computed will be clear from context, and we will then suppress the notational dependence on the function.

The most important result we need to invoke from \cite{MTT} is the following single-tree estimate.

\begin{proposition}\label{tree-est-prop}  Let $T$ be a regular tree, and let $f^1, f^2, f^3$ be test
functions on $\R$ with
\be{infty-bound}
\| f^i \|_\infty \le 1
\end{equation}
for all $1 \leq i \leq 3$.
Let $\pi_{\omega_{P_i}}$ be any operators as in \eqref{stmultiplerety}, with multipliers adapted to ${\omega_{P_i}}$; note that $\pi_{\omega_{P_i}}=\pi_{\omega_{P_i'}}$ whenever $\omega_{P_i}=\omega_{P_i'}$.

Then we have
\be{tree-est}
\bigg|\sum_{\pv \in T}
\int \chi_{I_\pv,j_\pv} \prod_{i=1}^3 \pi_{\omega_{P_i}} f^i\bigg| \lesssim_{\theta_2,\theta_3}
|I_T| \prod_{i=1}^3 \size^*_i(T)^{\theta_i}
\end{equation}
whenever $\theta_1 = 1$ and $0 < \theta_2,\theta_{3} < 1$.
\end{proposition}

We now briefly explain how to decompose  certain sub-trees of regular trees into regular sub-trees.
\begin{definition}\label{max-tree-def}
Fix $\mu \in \Z$. Consider a subset $\Pv$ of $\Pv_1^\mu$, some $\xi\in\R$, and some $J$-dyadic interval $I$. Then the \emph{maximal tree $T^*$ in $\Pv$ with top data $(\xi,I)$} is the set of all $\pv\in \Pv$ such that $I_\pv\subseteq I$ and $\overline{\omega}_{\xi,I}\subseteq \overline{\omega}_{P_i}$ for some $i\in\{1,2,3\}$. A tree in $\Pv$ is called \emph{maximal} if it is the maximal tree in $\Pv$ with some top data $(\xi,I)$.
\end{definition}

\begin{definition}
A \emph{tree selection process} consists of selecting a tree $T_1 \subset \Pv$, then continuing iteratively by selecting a tree at each stage from the complement of all the previously selected trees; that is, at the $k$-th stage we select a tree $T_k \subset \Pv\setminus (T_1\cup\ldots\cup T_{k-1})$. The trees $T_k$ chosen in such a process will be referred to as ``selected trees'' and their constituent multi-tiles as ``selected multi-tiles.''  A tree selection process is called \emph{greedy} if each selected tree $T_k$ is maximal in $\Pv \setminus (T_1\cup\ldots\cup T_{k-1})$.
\end{definition}

It is proven in Lemma 4.7 of \cite{MTT} that if $\Pv$ is regular, then all trees selected from $\Pv$ via a greedy selection process are regular. The same ideas can be used to prove:
\begin{lemma}
\label{convexunionforest}
Let $T_{k},T_{k+1},\ldots, T_{l}$ be consecutive trees selected from a regular collection $\Pv\subset \Pv_1^\mu$ via a greedy selection process. Define the collection of multi-tiles $S=T_{k}\cup \ldots\cup T_{l}$. Let $\pv,\pv'\in S$ and let $\pv''\in \Pv$ be such that $\pv'\le\pv''\le \pv$ and $|I_{\pv'}|<|I_{\pv''}|<|I_{\pv}|$.  Then $\pv''\in S$.
\end{lemma}
\begin{proof}
 Denote by $k\le s\le l$ the index such that $\pv\in T_s$. Note that $\pv''$ could not have been selected earlier than the tree $T_k$, since by greediness $\pv'$ should then have been selected at or before the same step. But, by the same reasoning, $\pv''$ must be selected at or before the $s$-th step since $\pv'' \leq \pv$.  Thus $\pv'' \in T_j$ for some $k \leq j \leq s$, and hence $\pv'' \in S$.

\end{proof}

An immediate consequence of Proposition \ref{tree-est-prop} and Lemma \ref{convexunionforest} is the following.
\begin{lemma}
\label{lema3:10}
Fix a regular collection $\Pv\subset \Pv_1^\mu$, and let $T\subset \Pv$ be a regular tree. Consider also two independent greedy selection processes from $\Pv$, with tree outcomes $T_1,T_2,\ldots$ and $T_1',T_2',\ldots$, respectively. Define the collections of multi-tiles $S=T_{k}\cup T_{k+1}\cup\ldots\cup T_{l}$ and $S'=T_{k'}'\cup T_{k'+1}'\cup\ldots\cup T_{l'}'$ for some $k\le l$ and $k'\le l'$. Define the subtree $\widetilde{T}=T\cap S\cap S'$ of $T$. Then
\be{sub-tree-est}
\bigg|\sum_{\pv \in \widetilde{T}}
\int \chi_{I_\pv,j_\pv} \prod_{i=1}^3 \pi_{\omega_{P_i}} f^i\bigg| \lesssim_{\theta_2,\theta_3}
|I_T| \prod_{i=1}^3 \size^*_i(T)^{\theta_i}
\end{equation}
whenever $\theta_1 = 1$ and $0 < \theta_2,\theta_{3} < 1$.
\end{lemma}
\begin{proof}
The tree $\widetilde{T}$ is not necessarily regular, but we will fix this issue as follows. Let $\pv^{(1)},\pv^{(2)},\ldots$ be the maximal (with respect to $\le$) multi-tiles in $\widetilde{T}$. For each $i$ define $$\widetilde{T}^{(i)}=\{\pv\in\widetilde{T}:\pv\le \pv^{(i)}\}.$$
Since the intervals $I_{\pv^{(i)}}$ are pairwise disjoint, the trees $\widetilde{T}^{(i)}$ partition $\widetilde{T}$. Lemma \ref{convexunionforest} shows that each $\widetilde{T}^{(i)}$ is regular; in short, a failure of regularity would entail the omission of multitiles playing the role of $\pv''$ in the lemma, while $\pv^{(i)} \in \widetilde T^{(i)}$ plays the role of $\pv$ and prohibits such an omission.  Applying Proposition \ref{tree-est-prop} to each $\widetilde{T}^{(i)}$, the result follows from the fact that $|\sum_{i}I_{\pv^{(i)}}|\le |I_T|$.

\end{proof}

\section{Size estimates}
\label{sec:size}
As noted earlier, the first component plays a special role since the 1-tiles have area 1. We will begin by deriving improved estimates over those in \cite{MTT} for $\size_1$.

\begin{proposition}
\label{netqw679679034or0-rjgnb}
We have
$$\size_1(T)\lesssim \sup_{\psi_{\pv}}\frac{1}{|I_T|^{1/2}}\bigg\|\Big(\sum_{\vec{P}\in T}\frac{|\langle f^1,\psi_{\pv}\rangle|^21_{I_{\pv}}}{|I_{\pv}|}\Big)^{1/2}\bigg\|_2+\inf_{x\in I_T} \mathcal M_1(f^1)(x),$$
where the supremum is taken over all $\psi_{\pv}$ such that $M_{(\xi_T)_1}\psi_{\pv}$ is $L^2$ adapted of order 1 to $I_{\pv}$ and has mean zero.
\end{proposition}
\begin{proof}
Write each
$$\pi_{m_{P_1}}f^1=\sum_{J\text{ dyadic}: \; |J|=|I_{\pv}|}\langle f^1,\phi_J^{(1)}\rangle \; \phi_J^{(2)}$$
where each $M_{(\xi_T)_1}\phi_J^{(i)}$ is $L^2$ adapted to $J$ of order 1 and has mean zero.

Next note that if the spatial parameter $N$ is sufficiently large, then
$$\Big\|\tilde \chi_{I_\pv}^{10}\sum_{J\text{ dyadic}:\; |J|=|I_{\pv}|}\langle f^1,\phi_J^{(1)}\rangle \; \phi_J^{(2)}\Big\|_2^2 \: \lesssim \: \sum_{J\text{ dyadic}: \; |J|=|I_{\pv}|}\bigg(\frac{1+\dist(J,I_{\pv})}{|I_{\pv}|}\bigg)^{-200}|\langle f^1,\phi_J^{(1)}\rangle|^2.$$
The conclusion will follow by noting that
$$\bigg(\frac{1+\dist(J,I_{\pv})}{|I_{\pv}|}\bigg)^{-100}M_{(\xi_T)_1}\phi_J^{(1)}$$
is $L^2$ adapted of order 1 to $I_{\pv}$ and has mean zero.

The estimate
$$|I_T|^{-\frac 12} \sup_{m_{i,T}}\big\|\tilde{\chi}_{I_T}^{10}\pi_{m_{i,T}}(f^i)\big\|_2 \; \lesssim \; \inf_{x\in I_T}\mathcal M_1(f^1)(x)$$
follows via similar considerations, together with the fact that (see \eqref{hdfuhuferu48990r746r83uy})
\begin{align*}
\inf_{x\in J}\mathcal M_1(f^1)(x) &\leq 100 \, \dist(I_T,J)^2 \inf_{x\in 10 \, \dist(I_T,J)J} \mathcal M_1(f^1)(x)\\
 &\leq 100 \, \dist(I_T,J)^2  \; \inf_{x\in I_T} \mathcal M_1(f^1)(x)
\end{align*}
for each interval $J$ such that $|J|=|I_T|$.
\end{proof}

We also recall the following abstract result, which is Lemma 4.2 from \cite{MTT3}.

\begin{lemma}
\label{lemabssizeL1}
Let $\I$ be a finite collection of dyadic intervals, and let $(a_I)_{I\in\I}$ be some complex numbers. Define the \textit{maximal size} of $\I$ by
$$\size^{*}(\I):=\sup_{J\textrm{ interval}}\bigg(\frac{1}{|J|}\sum_{I\in\I\atop{I\subset J}}|a_I|^2\bigg)^{1/2}=\sup_{J\textrm{ interval}}\frac{1}{|J|^{1/2}} \; \Big\|\Big(\sum_{I\in\I\atop{I\subset J}}|a_I|^2\frac{1_{I}}{|I|}\Big)^{1/2}\Big\|_{L^{2}}.$$
Then
$$\size^{*}(\I)\sim \sup_{J\text{ interval}}\frac{1}{|J|} \; \Big\|\Big(\sum_{I\in\I\atop{I\subset J}}|a_I|^2\frac{1_{I}}{|I|}\Big)^{1/2}\Big\|_{L^{1,\infty}}.$$
\end{lemma}

The following result is then an entirely standard corollary of Proposition \ref{netqw679679034or0-rjgnb}, Lemma \ref{lemabssizeL1} and the fact that
$$\Big\|\Big(\sum_{\vec{P}\in T}\frac{|\langle f^1,\psi_{\pv}\rangle|^21_{I_\pv}}{|I_\pv|}\Big)^{1/2}\Big\|_{1,\infty}\lesssim \|f^1\|_1.$$

\begin{corollary}
\label{est-sizebyM_1}
If $\Pv$ is any collection of multi-tiles, then
$$\size^*_{1}(\Pv)\lesssim \sup_{\vec{P}\in\Pv} \; \inf_{x\in I_{\pv}} \mathcal M_1f^1(x).$$
\end{corollary}

As is quite common, we will eventually excise certain ``exceptional sets'' $\Omega \subset \R$ in an appeal to generalized restricted type interpolation theory; the following result will provide size decay estimates for collections of multi-tiles with spatial intervals contained in the exceptional set.

\begin{proposition}
\label{size-exc}
Let  $\Omega\subset \R$, and let $f^3\in L^2(\R)$ be such that $|f^3|\lesssim 1_{\R\setminus\Omega}$.
For each  $l\ge 1$, let $\Pv_{1,l}$ consist of all the multi-tiles in $\Pv_1$ such that
$$4^{l}I_{\pv} \cap(\R\setminus\Omega)\not=\emptyset$$
while
$$4^{l-1}I_{\pv}\subset\Omega.$$
Then
$$\size^*_{3}(\Pv_{1,l})\lesssim_N 2^{-c(N)l},$$
where $$\lim_{N\to\infty}c(N)= \infty$$ and $N$ is the spatial parameter introduced earlier.
\end{proposition}
\begin{proof}

Let $T\in \Pv_{1,l}$ be a tree.  For each $\pv\in T$, write as in the proof of Proposition \ref{netqw679679034or0-rjgnb}
\begin{align*}
\big\| \tilde \chi_{I_\pv}^{10} \pi_{m_{P_3}}(f^3) \big\|_2^2 &\lesssim \sum_{J\text{ dyadic}: \; |J|=\frac{|I_\pv|}{1+s_\mu}}\Big(\frac{1+\dist(J,I_{\pv})}{|I_{\pv}|}\Big)^{-b(N)} \; \big|\langle f^3,\phi_J^{(1)}\rangle\big|^2\\
 &\lesssim 2^{-c(N) \, l}|I_{\pv}|,
\end{align*}
where $b(N),c(N)\to\infty$.
To see the last inequality, note that for the intervals $J\subset 2^{l-1}I_{\pv}$ the term $\big|\langle f^3,\phi_J^{(1)}\rangle\big|$ is small due to spatial support considerations. Also, if $J$ and $I_{\pv}$ are far apart then the term $\big(\frac{1+\dist(J,I_{\pv})}{|I_{\pv}|}\big)^{-b(N)}$ is small.

Finally, note that the spatial intervals $I_{\pv}$, $\pv\in T$, are pairwise disjoint. The result now follows immediately.
\end{proof}

Finally, we will need to recall the following result, which is Proposition 6.3 of \cite{MTT}.
\begin{proposition}
For each $i=1,2,3$ and each $f^i$ we have
$$\size^*_{i}(\Pv_{1})\lesssim \|f^i\|_{\infty}.$$
\end{proposition}

\section{Proof of Theorem \ref{main1}}

Let $E_1,E_2,E_3\subset \R$ be sets of finite measure. Because the claimed ranges of $p_2$ and $p_3$ are identical, we may assume without loss of generality that $|E_2|\le |E_3|$. By rescaling, we can also assume $|E_3|=1$.

Define the ``exceptional set''
$$\Omega:=\{x : \mathcal M_1(1_{E_1})>100|E_1|\},$$ and set
$$E_3':=E_3\setminus\Omega.$$
Note that $|E_3'|>\frac{|E_3|}{2}$, and thus via restricted type interpolation theory (see \cite{MTT4}) it suffices to prove that
$$\sum_{\mu} \Big|\sum_{\pv \in \Pv_1^{\mu}}
\int \chi_{I_\pv,j_\pv} \; {\pi}_{\omega_{P_1}} (M_{\xi_\mu}f^1_\mu) \; {\pi}_{\omega_{P_2}}(M_{\eta_\mu}f^2_\mu) \; {\pi}_{\omega_{P_3}}(M_{\theta_\mu}f^3_\mu) \Big| \: \lesssim \: |E_1|^{1/p_1}|E_2|^{1/p_2}$$
for each $|f^1|\le \chi_{E_1}$, $|f^2|\le \chi_{E_2}$, $|f^3|\le \chi_{E_3'}$.

Note that $M_{\xi_\mu}f^1_\mu$ is supported in frequency in the interval $J^1_\mu:=I_\mu^1-\xi_\mu$ containing the origin, and analogous frequency support properties hold for the other two components.  Recall the vector $\textbf{v}_\mu = (\textbf{v}_\mu^1 := 1, \textbf{v}_\mu^2:=s_\mu, \textbf{v}_\mu^3 :=-1-s_\mu)$; note that the frequency supports $J_\mu^i$ of the functions $M_{\xi_\mu}f^1_\mu$, $M_{\eta_\mu}f^2_\mu$, and $M_{\theta_\mu}{f}^3_\mu$ satisfy $|J_\mu^i|=K_\mu|\textbf{v}_\mu^i|$ for some unimportant constant $K_\mu$. This, together with the Whitney property \eqref{whitney1first} and the fact that the $J_\mu^i$ contain the origin, implies that in order for the product $$\pi_{\omega_{P_1}}(M_{\xi_\mu}f^1_\mu) \; \pi_{\omega_{P_2}}(M_{\eta_\mu}f^2_\mu) \; \pi_{\omega_{P_3}}(M_{\theta_\mu}f^3_\mu)$$ to be nonzero, one needs
\begin{equation}
\label{e12997gt}
|\omega_{P_i}|\lesssim_{C_0}|I_{\mu}^i|.
\end{equation}
This can indeed be seen easily by comparing $L_\mu^{-1}(\omega_{P_1}\times \omega_{P_2}\times\omega_{P_3})$ with $L_\mu^{-1}(J_\mu^1\times J_\mu^2\times J_\mu^3)$.  We can thus replace the collection $\Pv_1^{\mu}$ with the subcollection of multi-tiles $\pv$ satisfying \eqref{e12997gt} and such that $\omega_{P_i} \cap J_\mu^i \neq \emptyset$; we will continue to denote this collection by $\Pv_1^{\mu}$.

Let $\pi_\mu^1$ be the multiplier operator associated with the multiplier $\varphi_\mu^1(\xi+\xi_\mu)$, where we recall that $\widehat{f_{\mu}^i}:= \widehat{f^i}\varphi_\mu^i$.  Note that for each $f^1$ we have
\begin{equation}\label{samin3ways}
\pi_{\omega_{P_1}}(M_{\xi_\mu}f^1_{\mu})=\pi_{\omega_{P_1}}\pi_{\mu}^1(M_{\xi_\mu}f^1).
\end{equation}
Due to \eqref{e12997gt}, if $\vec{P}\in \Pv_1^{\mu}$, then $\pi_{\omega_{P_1}}\pi_{\mu}^1$ has the same essential properties as $\pi_{\omega_{P_1}}$, in that its multiplier is also adapted to $\omega_{P_1}$ (with a new but still universal constant).  An analogous observation applies to the second and third components.  Thus, the upshot of this discussion is that, for each $\mu$, the essential effect of restricting the $f^i$ in frequency to $I_\mu^i$ is simply a restriction on the collection of multi-tiles under consideration in the ``model sum.''

For each $\mu$ and for each $l\ge 1$, let $\Pv_{1,l}^{\mu}$ consist of all the tiles in $\Pv_1^{\mu}$ such that
$$4^{l}I_P\cap(\R\setminus\Omega)\not=\emptyset$$
and
$$4^{l-1}I_P\subset\Omega.$$
Also define $\Pv_{1,0}^{\mu}$ to consist of all the tiles in $\Pv_1^{\mu}$ such that
$$I_P\cap(\R\setminus\Omega)\not=\emptyset;$$
note that $(\Pv_{1,l}^{\mu})_{l\ge 0}$ forms a partition of $\Pv_{1}^{\mu}$.

Recall that each $\Pv_1^{\mu}$ is regular, and note that  $\Pv_{1,0}^{\mu}$ remains regular. While the $\Pv_{1,l}^{\mu}$ are not necessarily regular for $l\ge 1$, the property will not really be needed in this case; the combinatorics as much simpler for $l\ge 1$, since each tree in this context will turn out to consist only of one multi-tile.

So let us first treat the more difficult case of $l=0$.  We know from Section \ref{sec:size} that
$$\size_{1}^*(\Pv_{1,0}^{\mu})\lesssim |E_1|,$$
$$\size_{2}^*(\Pv_{1,0}^{\mu})\lesssim 1,$$
and
$$\size_{3}^*(\Pv_{1,0}^{\mu})\lesssim 1,$$
where the three sizes are computed with respect to the functions
$M_{\xi_\mu}f^1_\mu$, $M_{\eta_\mu}f^2_\mu$, and $M_{\theta_\mu}f^3_\mu$, respectively. Indeed, in view of \eqref{samin3ways} and the observation following it,
we easily get that
$$\size_{1}^*(\Pv_{1,0}^{\mu},M_{\xi_\mu}f^1_\mu)\approx \size_{1}^*(\Pv_{1,0}^{\mu},M_{\xi_\mu}f^1),$$
and similarly for $\size_{2}^*$ and $\size_{3}^*$.

Now, by iterating Proposition 9.2 of \cite{MTT} as in the proof of Corollary 9.3 (\textit{ibid.}),  we can decompose $\Pv_{1,0}^{\mu}$ in three ways as
\begin{align*}
\Pv_{1,0}^{\mu} &=\bigcup_{2^{-n_1}\lesssim |E_1|} \; \bigcup_{T\in\T_{n_1,\mu}^1}T,\\
\Pv_{1,0}^{\mu} &=\bigcup_{\, \, \; 2^{-n_2}\lesssim 1 \, \, \;} \; \bigcup_{T\in\T_{n_2,\mu}^2}T,\\
\Pv_{1,0}^{\mu} &=\bigcup_{\, \, \; 2^{-n_3}\lesssim 1 \, \, \;} \; \bigcup_{T\in\T_{n_3,\mu}^3}T.
\end{align*}
The collections $\T_{n_j,\mu}^j$ are referred to as ``forests''; their specific construction is detailed in Sections 9 and 10 of \cite{MTT}.  In particular, they enjoy the maximal $j$-size bounds
\begin{equation}\label{forestsizebound}
\size_{j}^*\big(\T_{n_j,\mu}^j\big) \lesssim 2^{-n_j}.
\end{equation}
Additionally, we note that each of the trees $T$ appearing in the above decompositions is regular, as guaranteed by the greedy selection process in \cite{MTT} and by the regularity of the original collection $\Pv_{1,0}^{\mu}$. We also have the ``local''\footnote{``Locality'' here is understood with respect to the parameter $\mu$.} Bessel inequalities
$$\sum_{T\in \T_{n_1,\mu}^1}|I_T|\lesssim 2^{2n_1}\|f_\mu^1\|_2^2,$$
$$\sum_{T\in \T_{n_2,\mu}^2}|I_T|\lesssim 2^{2n_2}\|f_\mu^2\|_2^2,$$
$$\sum_{T\in \T_{n_3,\mu}^3}|I_T|\lesssim 2^{2n_3}\|f_\mu^3\|_2^2.$$

Now, the key place where the disjointness of the intervals $I_\mu^i$ is exploited is naturally in the fact that the functions $f_\mu^i$ are pairwise orthogonal, for any fixed $i$. Thus, if we define
$$\T_{n_i}^i := \bigcup_{\mu}\T_{n_i,\mu}^i \;,$$
we  obtain the crucial ``global'' Bessel inequality
\begin{equation}
\label{globalbessel}
\sum_{T\in \T_{n_i}^i}|I_T|\lesssim 2^{2n_i}|E_i|.
\end{equation}

Let $S_{n_i,\mu}^i$ be the collection of all the multi-tiles from all the trees in the forest $\T_{n_i,\mu}^i$, and define $$S_{n_1,n_2,n_3,\mu}=\bigcap_{i=1}^3S_{n_i,\mu}^i.$$  For each  $T\in \T_{n_i,\mu}^i$, define $$T':=T\bigcap_{j\in\{1,2,3\}\setminus \{i\}} S_{n_j,\mu}^j = T \cap S_{n_1,n_2,n_3,\mu}$$
and note that, for each $i$, these collections $T'$ partition the set $\S_{n_1,n_2,n_3,\mu}$.

Now choose $0<\gamma_2,\gamma_3<1/2$ such that $\gamma_2+\gamma_3=\frac1p_1$ and $\gamma_2>\frac1{p_2}$. Next, choose $1>\theta_2>2\gamma_2$ and  $1>\theta_3>2\gamma_3$, and set $\theta_1=1$.  By invoking Lemma \ref{lema3:10} for the $T'$ and applying (\ref{forestsizebound}) and (\ref{globalbessel}),
we obtain for each $i$
\begin{align*}
&\sum_\mu \Big| \sum_{\pv \in \S_{n_1,n_2,n_3,\mu}}
\int \chi_{I_\pv,j_\pv} \; {\pi}_{\omega_{P_1}}(M_{\xi_\mu}f^1_\mu)\; {\pi}_{\omega_{P_2}}(M_{\eta_\mu}f^2_\mu) \; {\pi}_{\omega_{P_3}}(M_{\theta_\mu}f^3_\mu) \; \Big|\\
\lesssim \; &\sum_\mu\sum_{T\in \T_{n_i,\mu}^i}\; \Big |\sum_{\pv \in T'}
\int \chi_{I_\pv,j_\pv}  \; {\pi}_{\omega_{P_1}}(M_{\xi_\mu}f^1_\mu) \; {\pi}_{\omega_{P_2}}(M_{\eta_\mu}f^2_\mu) \; {\pi}_{\omega_{P_3}}(M_{\theta_\mu}f^3_\mu) \; \Big|\\
\lesssim \: \: &2^{-n_1} \; 2^{-n_2\theta_2} \; 2^{-n_3\theta_3} \; 2^{2n_i}|E_i|.
\end{align*}
By taking a geometric mean, the above is further bounded by $$2^{-n_1} \; 2^{-n_2\theta_2} \; 2^{-n_3\theta_3} \; (2^{2n_1}|E_1|)^{1-\frac1{p_1}} \; (2^{2n_2}|E_2|)^{\gamma_2} \; (2^{2n_3}|E_3|)^{\gamma_3}.$$  Finally, note that for each $\mu$ the sets $S_{n_1,n_2,n_3,\mu}$ form a partition of $\Pv_{1,0}^{\mu}$, and thus
\begin{align*}
 & \phantom{\lesssim_{p_1,p_2,p_3}} \:\: \sum_{\mu} \; \Big| \sum_{\pv \in \Pv_{1,0}^{\mu}}
\int \chi_{I_\pv,j_\pv} \; {\pi}_{\omega_{P_1}}(M_{\xi_\mu}f^1_\mu) \; {\pi}_{\omega_{P_2}}(M_{\eta_\mu}f^2_\mu) \; {\pi}_{\omega_{P_3}}(M_{\theta_\mu}f^3_\mu) \; \Big|\\
&\lesssim\phantom{_{p_1,p_2,p_3}} \sum_{2^{-n_1}\lesssim |E_1|} \; \sum_{2^{-n_2}\lesssim 1} \; \sum_{2^{-n_3}\lesssim
1} \; 2^{-n_1} \; 2^{-n_2\theta_2} \; 2^{-n_3\theta_3} \; (2^{2n_1}|E_1|)^{1-\frac1{p_1}} \; (2^{2n_2}|E_2|)^{\gamma_2} \; 2^{2n_3\gamma_3}\\
&\lesssim_{p_1,p_2,p_3} \;|E_1|^{1/p_1}|E_2|^{1/p_2},
\end{align*}
as desired.

Now for $l \geq 1$ the corresponding sums over the multi-tiles in $\Pv_{1,l}^{\mu}$ will have an extra geometric decay in $l$, due to Proposition \ref{size-exc}. More precisely, via Corollary \ref{est-sizebyM_1}
and the classical estimate
\begin{equation}
\label{hdfuhuferu48990r746r83uy}
\inf_{x\in I} \mathcal M_1(f)(x)\lesssim 8^l\inf_{x\in 4^lI} \mathcal M_1(f)(x),
\end{equation}
we obtain
$$\size_{1}^*(\Pv_{1,l}^{\mu})\lesssim 8^l|E_1|.$$
If $N$ is large enough, this will be compensated by
$$\size_{3}^*(\Pv_{1,l}^{\mu})\lesssim 2^{-c(N)l}.$$
We will continue to be happy with
$$\size_{2}^*(\Pv_{1,l}^{\mu})\lesssim 1.$$

Running the tree selection process as in \cite{MTT}, all the selected trees from $\Pv_{1,l}^{\mu}$ will now turn out to consist of just one multi-tile; they are thus regular, and Proposition \ref{tree-est-prop} applies.  What allows us to work only with singleton trees is the fact that the multi-tiles in $\Pv_{1,l}^{\mu}$ are automatically ``strongly disjoint'' in the sense of Lemma 4.16 of \cite{MTT}, since  spatial intervals of such multi-tiles with distinct scales are pairwise disjoint. The details are then the same as in the case $l=0$.

\section{Proof of Theorem \ref{main2}}

We can assume that the lines $l_\mu$ intersect the rectangles $I^{1}_\mu\times I_\mu^2$. Indeed, if  the rectangle and the line do not intersect for some $\mu$, then one can easily check that $$\Lambda_\mu(f^1,f^2,f^3)=c_\mu\int f^1_\mu(x) f^2_\mu(x){f}^3_\mu(x)dx,$$
for some $c_\mu=O(1)$. The boundedness of $\sum_\mu|\Lambda_\mu|$, where the sum is taken over the $\mu$ for which intersection fails, follows via a much simpler version of the main argument. For each such $\mu$, all trees will have the same frequency component; they will only differ spatially.

Under the assumption that $l_\mu$ intersects the rectangle $I^{1}_\mu\times I_\mu^2$, for each $\mu$ we can choose $(\xi_\mu,\eta_\mu,\theta_\mu)\in l_\mu^3$ such that in addition $\xi_\mu\in I_{1}^\mu$,  $\eta_\mu\in I_{2}^\mu$ and $\theta_\mu\in I_{3}^\mu$.

Define
$$\Lambda_\mu^{*}(f^1,f^2,f^3)=\int H_{s_\mu}(f^1,f^2){f}^3\; ;$$
then
$$\Lambda_\mu(f^1,f^2,f^3)=\Lambda_\mu^*\big(M_{\xi_\mu}f^1_\mu \, , \, M_{\eta_\mu}f^2_\mu \, , \, M_{\theta_\mu}{f}^3_\mu\big).$$
In the first three sections of \cite{MTT} it is proved that $|\Lambda_\mu^{*}(f^1,f^2,f^3)|$ is bounded from above by a sum of boundedly many model sums of the form
\begin{equation}
\Big|\sum_{\pv \in \Pv_1^{\mu}}
\int \chi_{I_\pv,j_\pv} \prod_{i=1}^3 \pi_{\omega_{P_i}} f^i\Big|.
\end{equation}
Theorem \ref{main2} now follows immediately from Theorem \ref{main1}.

\section{Proof of Theorem \ref{main3}}

We will decompose $P_{\operatorname{lac}}$ similarly to the decomposition of the disc in \cite{GL}, with a few simplifications along the way. Our decomposition is motivated by the following principle.

\begin{proposition}
\label{partitionunity}
Let $\Omega\subset\R^2$ be an open set, and let $0<\alpha<1$, $M_1,M_2>0$. Let
$\mathcal C$ be a collection of closed rectangles $R=I_R\times J_R$ such that:
\begin{enumerate}

\item $ R \subset\Omega$ for each $R\in\mathcal C$.

\item $\Omega=\bigcup_{R\in \mathcal C}\alpha R$.

\item For  each  $(x,y)\in\Omega$ there are at most $M_1$  rectangles $R$ containing $(x,y)$.

\item If $R$ and $R'$ intersect, then
$$\frac1{M_2}<\frac{|I_R|}{|I_{R'}|},\frac{|J_R|}{|J_{R'}|}<M_2.$$
\end{enumerate}

Then there exist  functions $\psi_R$ adapted to and supported on $R$ such that $$\sum_{R\in \mathcal C} \psi_R=\chi_\Omega.$$
The adaptation constant will only depend on $\alpha$, $M_1$ and $M_2$.
\end{proposition}
\begin{proof}
For each $R \in \mathcal C$, first choose some $\eta_R$ adapted and supported on $R$ such that $\chi_{\alpha R}\le \eta_R\le \chi_{R}$.
It is easy to check that
$$\psi_R=\frac{\eta_R}{\sum_{R'\in \mathcal C}\eta_{R'}}$$
satisfies the desired properties.

\end{proof}

We first decompose $P_{\operatorname{lac}}$ into triangles. Each triangle has one vertex at the origin and the other two vertices are two consecutive vertices of $P_{\operatorname{lac}}$. Note that while the line segment
$$\{(\xi,0) : -1\le \xi\le 1\}$$
is part of the polygon and not part of any of the triangles, this will not affect our argument since the line segment has two-dimensional measure zero.

Due to symmetry, it will suffice to focus on the part of $P_{\operatorname{lac}}$ lying in the second quadrant. For each integer $\mu\ge 1$, let $U_\mu$ be the triangle as above with vertices at the origin, at $v_\mu:=(\cos(\pi-\pi2^{-\mu}),\sin(\pi-\pi2^{-\mu}))$, and at $v_{\mu+1}$. Note that the slope $s_\mu$ of the line segment $l_\mu$ joining $v_\mu$ and $v_{\mu+1}$ satisfies $s_\mu\approx 2^{\mu}$.
For each $\mu\ge 1$, let $T_{\mu}$ be the trapezoidal region inside $U_\mu$ trapped between the dilate $(1-2^{-2\mu})P_{\operatorname{lac}}$ and $P_{\operatorname{lac}}$ itself.  (These are the shaded regions of  Figure \ref{trapezoidfigure}.)

\begin{figure}
\begin{tikzpicture}[scale=1.25]
\begin{scope}[cm={-1,0,0,1,(0,0)}]
%\draw [->] (0,0) -- (0,4.5);
\draw [->] (0,0) -- (4.5,0);
\draw [help lines] (4,0) arc (0:45:4);
\draw [thick] (45:4) -- (22.5:4) -- (11.25:4)--(5.625:4)--(2.8125:4)--(0:4);
\draw [help lines] (0,0) -- (45:4);
\draw [help lines] (0,0) -- (22.5:4);
\draw [help lines] (0,0) -- (11.25:4);
\draw [help lines] (0,0) -- (5.625:4);
\fill (4.25:3.75) circle (.4pt);
\fill (4.25:3.75) ++(-88:.08) circle (.4pt);
\fill (4.25:3.75) ++(-88:.16) circle (.4pt);
\fill (47:3) circle(.4pt);
\fill (47:3) ++(-.05,.05)  circle(.4pt);
\fill (47:3) ++(-.1,.1) circle(.4pt);
\draw [fill=gray!20] (45:3) -- (45:4) -- (22.5:4) -- (22.5:3) -- cycle;
\draw (33.75:3.4) node[font=\footnotesize]{$T_\mu$};
\draw [fill=gray!20] (22.5:3.75) -- (22.5:4) -- (11.25:4) -- (11.25:3.75) -- cycle;
\draw (16.35:3) node[font=\footnotesize]{$U_{\mu+1}$};
\draw [fill=gray!20] (11.25:3.9375) -- (11.25:4) -- (5.625:4) -- (5.625:3.9375) -- cycle;
\fill (45:4) circle(1.25pt);
\draw (45:4) node [font=\footnotesize, left] {$v_\mu$};
\fill (22.5:4) circle(1.25pt);
\draw (22.5:4) node [font=\footnotesize, left] {$v_{\mu+1}$};
\fill (11.25:4) circle(1.25pt);
\draw (11.25:4) node [font=\footnotesize, left] {$v_{\mu+2}$};
\fill (5.625:4) circle(1.25pt);
\draw (5.625:4) node [font=\footnotesize, left] {$v_{\mu+3}$};
\end{scope}
\end{tikzpicture}
\caption{Trapezoids $T_\mu$ in the second quadrant.} \label{trapezoidfigure}
\end{figure}
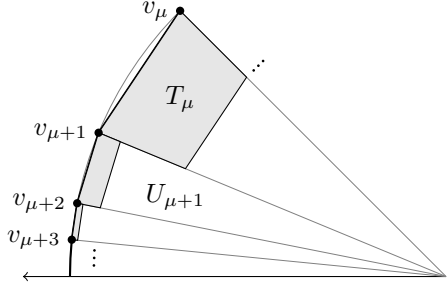

We will now cover each region $T_\mu$ by certain ``Whitney rectangles.''  To that end, let $\D_2$ be the collection of all cubes in $\R^2$ of side-length $2^{j}$ for some integer $j$ whose centers lie in the lattice $2^{j-10} \Z^2$.  Let ${\bf \tilde{S}}$ be the collection of all squares $S$ in $\D_2$ such that
\begin{equation}
\label{dcjhuierygivbrtrc3oput89t9u689sgdf5q32532-=}
C_0S\cap\{(\xi,\xi),\xi\in\R\}=\emptyset,
\end{equation}
$$4C_0S\cap\{(\xi,\xi),\xi\in\R\}\not=\emptyset.$$
Note that
\begin{equation}
\label{dcjhuierygivbrtrc3oput89t9u689}
(\R\times\R)\setminus \big\{(\xi,\xi),\xi\in\R\big\}=\bigcup_{S\in {\bf \tilde{S}}}\frac12 \, S.
\end{equation}

Now consider the linear transformation $L_\mu^2$ on $\R^2$ given by
$$L_\mu^2(\xi,\theta)=(-\xi,-s_\mu\theta).$$ Let $\alpha<1$ be sufficiently close to 1 ($\alpha=1-10^{-10}$ will probably be enough).

Let ${\bf \tilde{S}_\mu}$ be the collection of all the squares  in ${\bf \tilde{S}}$  such that the rectangle $\alpha R_{S,\mu}$ intersects $T_{\mu}$, where $$R_{S,\mu}:=v_\mu+L_\mu^2(S);$$ see Figure \ref{covertmufigure}.
Note that due to \eqref{dcjhuierygivbrtrc3oput89t9u689}, the rectangles $\alpha \, R_{S,\mu}$ with $S\in {\bf \tilde{S}_\mu}$ cover $T_{\mu}\setminus l_\mu$. Also, due to  \eqref{dcjhuierygivbrtrc3oput89t9u689sgdf5q32532-=}, if $C_0$ is sufficiently large then each $R_{S,\mu}$ will lie inside $P_{\operatorname{lac}}$.

\begin{figure}
\begin{tikzpicture}[scale=1.25]
\begin{scope}[cm={-1,0,0,1,(0,0)}]
%\draw [->] (0,0) -- (0,4.5);
%\draw [->] (0,0) -- (4.5,0);
\draw [help lines] (11.25:4) arc (11.25:33.75:4);
\draw [help lines] (0,0) -- (33.75:4);
\draw [help lines] (0,0) -- (11.25:4);
\draw [thick, fill=gray!10] (33.75:2.5) -- (33.75:4) -- (11.25:4) -- (11.25:2.5) -- cycle;
\fill (33.75:4) circle(1.25pt);
\draw (33.75:4) node [font=\footnotesize, left] {$v_\mu$};
\fill (11.25:4) circle(1.25pt);
\draw (11.25:4) node [font=\footnotesize, left] {$v_{\mu+1}$};
\begin{scope} [yshift=-.75mm, xshift=.45mm]
\draw (2.8,.3) -- (2.8,1.8);
\draw (2.8,1.8) -- (2.2,1.8);
\draw (2.95,.3) --(2.95,1.8);
\draw (2.5,.3) -- (2.5,2.3);
\draw (2.65,.3) -- (2.65,2.3);
\draw (2.65,2.3) -- (2.5,2.3);
\draw (2.65,2.05) -- (2.8,2.05) -- (2.8,1.8);
\draw (2.725,1.8) -- (2.725,2.05);
\draw (2.2,.3) -- (3.1,.3) -- (3.1,1.3);
\draw (3.1,.8) -- (2.2,.8);
\draw (2.2,.3) -- (2.2,2.3);
\draw (2.35,.3) -- (2.35,1.8);
\draw (3.1,1.3) -- (1.9,1.3);
\draw (3.1,.3) -- (3.4,.3) -- (3.4,.8) -- (3.1,.8);
\draw (3.25,.3) -- (3.25,1.3);
\draw (3.4,.8) -- (3.4,1.3) -- (3.1,1.3) -- (3.1,1.8) -- (2.8,1.8);
\draw (3.4,1.05) -- (3.1,1.05);
\draw (3.1,1.55) -- (2.8,1.55);
\draw (3.325,.8) -- (3.325,1.3);
\draw (3.175,.8) -- (3.175,1.3);
\draw (3.025,1.3) -- (3.025,1.8);
\draw (2.875,1.3) -- (2.875,1.8);
\draw (3.4,.8) -- (3.55,.8) -- (3.55,.3) -- (3.4,.3);
\draw (2.2,.3) -- (1.9,.3) -- (1.9,2.3) -- (2.2,2.3);
\end{scope}
\end{scope}
\end{tikzpicture}
\caption{Covering $T_\mu$ by ``Whitney rectangles.''} \label{covertmufigure}
\end{figure}
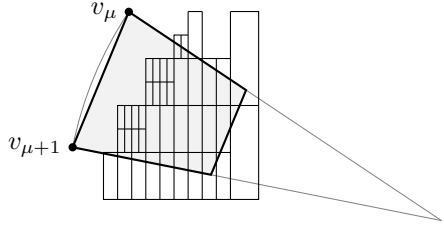

Define
\begin{align*}
{\bf \tilde{S}}^1 &=\bigcup_{\mu\ge 1}{\bf \tilde{S}_{\mu}},\\
{\bf \tilde{R}_\mu}&=\{R_{S,\mu}:S\in {\bf \tilde{S}_{\mu}}\},\\
{\bf \tilde{R}}^1 &=\bigcup_{\mu\ge 1}{\bf \tilde{R}_\mu},
\end{align*}
and for $i=2,3,4$ analogously define the collections ${\bf \tilde{S}}^i$ and ${\bf \tilde{R}}^i$ from the remaining quadrants. Set $${\bf \tilde{S}}=\bigcup_{i=1}^4{\bf \tilde{S}}^i \:\: ,\;\;\;\;\;\;{\bf \tilde{R}}=\bigcup_{i=1}^4{\bf \tilde{R}}^i.$$

For $\mu\ge 1$, define the intervals
\begin{align*}
J_\mu^1&:=\bigcup_{R=J_R^1\times J_R^2\in {\bf \tilde{R}_\mu}}J_R^1\:,\\
J_\mu^2&:=\bigcup_{R=J_R^1\times J_R^2\in {\bf \tilde{R}_\mu}}J_R^2\:,\\
J_\mu^3&:=\bigcup_{R=J_R^1\times J_R^2\in {\bf \tilde{R}_\mu}}(-J_R^1-J_R^2)\;,
\end{align*}
and the intervals $I_\mu^i :=\frac1\alpha \, J_\mu^i$. Observe that, crucially, if $\alpha$ is small enough then for each $i$ the intervals $(I_{\mu}^i)$ have bounded overlap.

Define
${\bf \tilde{R}}_1'$ to consist of all the rectangles
$$R_\mu:=[-(1-2^{-2\mu})\cos(\pi 2^{-\mu-1}),-(1-2^{-2\mu+2})\cos(\pi 2^{-\mu})]\times[0,(1-2^{-2\mu})\sin(\pi 2^{-\mu})]$$
for $\mu\ge 2$; see Figure \ref{paraproductfigure}.

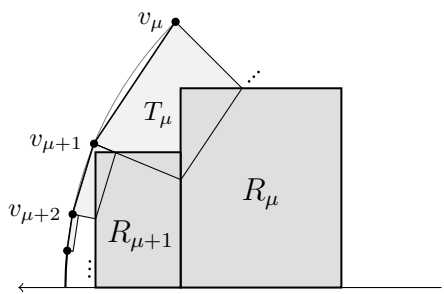
\begin{figure}
\begin{tikzpicture}[scale=1.25]
\begin{scope}[cm={-1,0,0,1,(0,0)}]
%\draw [->] (0,0) -- (0,4.5);
\draw [->] (0,0) -- (4.5,0);
\draw [help lines] (4,0) arc (0:45:4);
\draw [thick] (45:4) -- (22.5:4) -- (11.25:4)--(5.625:4)--(2.8125:4)--(0:4);
%\draw [help lines] (0,0) -- (45:4);
%\draw [help lines] (0,0) -- (22.5:4);
%\draw [help lines] (0,0) -- (11.25:4);
%\draw [help lines] (0,0) -- (5.625:4);
\fill (4.25:3.75) circle (.4pt);
\fill (4.25:3.75) ++(-88:.08) circle (.4pt);
\fill (4.25:3.75) ++(-88:.16) circle (.4pt);
\fill (47:3) circle(.4pt);
\fill (47:3) ++(-.05,.05)  circle(.4pt);
\fill (47:3) ++(-.1,.1) circle(.4pt);
\draw [fill=gray!10] (45:3) -- (45:4) -- (22.5:4) -- (22.5:3) -- cycle;
\draw (33.75:3.4) ++(.18,-.05) node[font=\footnotesize]{$T_\mu$};
\draw [fill=gray!10] (22.5:3.75) -- (22.5:4) -- (11.25:4) -- (11.25:3.75) -- cycle;
%\draw (16.35:3) node[font=\footnotesize]{$U_{\mu+1}$};
\draw [fill=gray!10] (11.25:3.9375) -- (11.25:4) -- (5.625:4) -- (5.625:3.9375) -- cycle;
\fill (45:4) circle(1.25pt);
\draw (45:4) node [font=\footnotesize, left] {$v_\mu$};
\fill (22.5:4) circle(1.25pt);
\draw (22.5:4) node [font=\footnotesize, left] {$v_{\mu+1}$};
\fill (11.25:4) circle(1.25pt);
\draw (11.25:4) node [font=\footnotesize, left] {$v_{\mu+2}$};
\fill (5.625:4) circle(1.25pt);
%\draw (5.625:4) node [font=\footnotesize, left] {$v_{\mu+3}$};
\draw [fill=gray!25, line width=.75pt] (1.065,0) -- (1.065,2.121) -- (2.775,2.121) -- (2.775,0) -- cycle;
\draw [fill=gray!25, line width=.75pt] (2.775,0) -- (2.775,1.44) -- (3.68,1.44) -- (3.68,0) -- cycle;
\draw (45:3) -- (22.5:3) -- (22.5:4);
\draw (22.5:3.75) -- (11.25:3.75);
\draw (1.92,1) node {$R_\mu$};
\draw (3.2,.55) node {$R_{\mu+1}$};
\end{scope}
\end{tikzpicture}
\caption{``Paraproduct'' rectangles $R_\mu$ in the second quadrant.} \label{paraproductfigure}
\end{figure}

Analogously, define the collections ${\bf \tilde{R}}_i'$, $i=2,3,4$, from the remaining quadrants. Let $${\bf \tilde{R}'}=\bigcup_{i=1}^4{\bf \tilde{R}}_i'\:\:,\;\;\;\;\;\;{\bf \tilde{R}''}=\left\{ \frac1\alpha R \: : \: R \in \mathbf{\tilde R}'\right\}.$$

We leave to the reader the verification of the following key geometric fact.
\begin{proposition}
The collection $\mathcal C$ consisting of the square
$$S_0=\left[-\frac{\sqrt{2}}{2}\;,\; \frac{\sqrt{2}}{2}\right]\times \left[-\frac{\sqrt{2}}{2}\; ,\; \frac{\sqrt{2}}{2}\right]$$
 together with all the rectangles in ${\bf \tilde{R}}$ and ${\bf \tilde{R}''}$ satisfies the hypotheses of Proposition \ref{partitionunity}, with $\Omega$ being the interior of $P_{\operatorname{lac}}$, the $\alpha$ chosen before, and some sufficiently large $M_1,M_2$.
\end{proposition}
Thus, there exist functions $\psi_R$ adapted to and supported on rectangles $R$ such that
$$\chi_{P_{\operatorname{lac}}}=\psi_{S_0}+\sum_{R\in {\bf \tilde{R}}}\psi_{R}+\sum_{R\in {\bf \tilde{R}''}}\psi_{R}.$$

The bilinear Fourier multiplier with symbol $\psi_{S_0}$  trivially maps $L^{p_1}\times L^{p_2}$ to $L^{p_3'}$ for each $1\le p_1,p_2\le \infty$. The multiplier corresponding to the sum over $R\in {\bf \tilde{R}''}$ is essentially a ``quadratically scaled paraproduct'' and is bounded using Theorem 2.2 in \cite{Li}; we include the proof of a simplified case in the appendix to this paper.

We thus focus on the main term $\sum_{R\in {\bf \tilde{R}}}\psi_{R}$.  Recall the collection  $\tilde \Q$ of Whitney cubes introduced earlier.  For $\mu\ge 1$ define the linear transformation $L_\mu$ on $\R^3$ given by $$L_\mu(\xi,\eta,\theta)=(-\xi,-s_\mu\eta,(1+s_\mu)\theta),$$ let $$\Q_\mu=\{L_\mu(\tilde Q):\tilde Q\in \tilde \Q\},$$ and let $R\in {\bf \tilde{R}}$. We assume without any loss that $R$ lies in the second quadrant, more precisely that $R\in {\bf \tilde{R}_\mu}$ for some $\mu\ge 1$.

Let $v_\mu=(\xi_\mu,\eta_\mu)$, and define $\theta_\mu=-\xi_\mu-\eta_\mu$.

Let $S=I\times J$ be the  square in ${\bf\tilde{S}_\mu}$ such that $R=L_\mu^2(S)+v_\mu$. Then $L_\mu^2(S)=(-I)\times (-s_\mu\cdot J).$ If $(\xi,\eta)\in L_\mu^2(S)$ then $\theta=-\xi-\eta\in K_R:=I+s_\mu\cdot J$.
Let $\psi_{K_R}$ be adapted to $\frac1\alpha K_R$ and satisfy
$$\chi_{K_R}\le \psi_{K_R}\le \chi_{\frac1\alpha K_R}.$$

Let $\xi_0\in 4C_0I\cap 4C_0J$. It is easy to check that
$$\frac1{\alpha}K_R\subset \bigcup_{\substack{K\in D^1, \\ \xi_0\in 10 \, C_0 \,K}}{\alpha}\big((1+s_\mu)\cdot K\big).$$
Also, note that for each $K$ as above we have $I\times J\times K\in \tilde\Q$.

Using these last two observations we can write
$$\psi_{K_R}=\sum_{\omega_3:(R-v_\mu)\times \omega_3\in {\bf Q_\mu}}m_{\omega_3}$$
for some multipliers $m_{\omega_3}$ adapted to and supported on the intervals $\omega_3$.

We conclude that for each $f^i$
\begin{align*}
&\int_{\xi+\eta+\theta=0}\hat{f^1}(\xi)\hat{f^2}(\eta)\hat{f^3}(\theta)\sum_{R\in {\bf \tilde{R}_\mu}}\psi_{R}(\xi,\eta)d\xi d\eta d\theta\\
=&\int_{\xi+\eta+\theta=0}\hat{f^1_\mu}(\xi)\hat{f^2_\mu}(\eta)\hat{f_\mu^3}(\theta)\sum_{R\in {\bf \tilde{R}_\mu}}\psi_{R}(\xi,\eta)\psi_{K_R}(\theta-\theta_\mu)d\xi d\eta d\theta\\
=&\sum_{Q=\omega_1\times \omega_2\times \omega_3\in {\bf Q_\mu}}\int\pi_{\omega_1}(M_{\xi_\mu}f_{\mu}^1)\pi_{\omega_2}(M_{\eta_\mu}f_{\mu}^2)\pi_{\omega_3}(M_{\theta_\mu}f_{\mu}^3),
\end{align*}
where as before $\pi_\omega$ is a frequency multiplier adapted to and supported on the interval $\omega$.  Discretizing in space at each scale via the partitions of unity (\ref{spatialdiscret}) and summing in $\mu \geq 1$, we arrive at a model sum that can be bounded by Theorem \ref{main1}; by duality, this completes the proof of Theorem \ref{main3}.

\section{Appendix: Quadratic paraproducts}

We outline a simple proof of a version of Theorem 2.1 in \cite{Li}. This is meant to help the reader understand the main ideas in a simplified setting.

\label{para}
Let $$Q_k(f)=(\hat{f}1_{[-2^{-k},-2^{-k-1}]\cup[2^{-k-1},2^{-k}]})\check {\;}$$
$$P_k(f)=(\hat{f}1_{[-2^{-k},2^{-k}]})\check {\;}$$
\begin{theorem}
For each  $1<p_1,p_2,p_3<\infty$ satisfying
$$\frac1{p_1}+\frac1{p_2}+\frac1{p_3}=1$$
we have
$$\Big\|PP(f,g):=\sum_{k}Q_{2k}(f)P_k(g)\Big\|_{p_3'}\lesssim\|f\|_{p_1}\|g\|_{p_2}.$$
\end{theorem}
\begin{proof}
We use the classical telescoping argument. Note that due to considerations of frequency support we have
\begin{align*}
\langle PP(f,g),h\rangle &= \int\sum_kQ_{2k}(f)P_k(g)P_{k-1}(h)\\
&=\int \sum_kQ_{2k}(f)\Big(g-\sum_{l<k}Q_l(g)\Big)\Big(h-\sum_{s<k-1}Q_s(h)\Big)\\
&=\int \sum_kQ_{2k}(f)gh - \int \sum_kQ_{2k}(f)\Big(\sum_{l<k}Q_l(g)\Big)h - \int \sum_kQ_{2k}(f)\Big(\sum_{s<k-1}Q_s(h)\Big)g\\
&\phantom{=\sum_kQ_{2k}(f)gh}+\int \sum_{l}Q_l(g)Q_l(h)\sum_{k>l}Q_{2k}(f) + \int \sum_{l}Q_{l}(g)Q_{l-1}(h)\sum_{k>l}Q_{2k}(f)\\
&\phantom{=\sum_kQ_{2k}(f)gh}+\int \sum_{l}Q_{l}(g)Q_{l+1}(h)\sum_{k>l+1}Q_{2k}(f).
\end{align*}
To ``diagonalize'' the sum above we used the frequency support of each $Q_k$. Finally, use H\"older, the boundedness of the square function $$S(\psi)=\Big(\sum_k|Q_k(\psi)|^2\Big)^{1/2}$$ and of the maximal martingale transform
$$\Big\|\sup_{l}\Big|\sum_{k>l}a_k \, Q_k(\psi)\Big| \, \Big\|_p\lesssim \|a_k\|_{\infty}\|\psi\|_p.$$
For example,

\begin{align*}
& \; \Big|\sum_k Q_{2k}(f) \, \Big(\sum_{l<k}Q_l(g)\Big) \, h\Big|\\
=& \; \Big|h\sum_l Q_l(g)\sum_{k>l}Q_{2k}(f)\Big|\\
=& \; \Big|\sum_l Q_l(g)\sum_{k>l}Q_{2k}(f) \, \big(Q_{l-1}h+Q_l(h)+Q_{l+1}(h)\big)\Big|\\
\leq & \; 3 \, S(f) \, S(h)\, \sup_{l}\Big|\sum_{k>l}Q_{2k}(g)\Big|.
\end{align*}
\end{proof}

\end{document}